\documentclass[12pt, letterpaper]{amsart}

\usepackage{amsxtra}
\usepackage{graphicx}
\usepackage{amscd}
\usepackage[german,english]{babel}
\usepackage{amsmath,amsthm}
\usepackage{amssymb}
\usepackage{enumerate}

\usepackage{textcomp, verbatim,mathrsfs, braket}

\usepackage[usenames,dvipsnames]{color}
\usepackage[colorlinks=true, pdfstartview=FitV,
 linkcolor=blue,citecolor=blue,urlcolor=blue]{hyperref}

\newcommand{\nc}{\newcommand}

\newenvironment{jaune}{\relax\color{Orchid}}{\hspace*{.5ex}\relax}

\newcommand{\bjr}{\begin{jaune}}
\newcommand{\jr}{\end{jaune}}

\nc{\bjn}{\begin{jaune}}
\nc{\ejn}{\end{jaune}}
\nc{\KLR}{Khovanov-Lauda-Rouquier algbera}
\nc{\KLRs}{Khovanov-Lauda-Rouquier algberas}

 \newtheorem{thm}{Theorem}[section]

 \newtheorem{co}[thm]{Corollary}
 \newtheorem{lem}[thm]{Lemma}
 \newtheorem{prop}[thm]{Proposition}
 
  \theoremstyle{definition}
 \newtheorem{defn}[thm]{Definition}
  
  \newtheorem{defn-thm}[thm]{Definition-Theorem}
 \theoremstyle{remark}

 \newtheorem{rem}[thm]{Remark}

\numberwithin{equation}{section}
\nc{\Lemma}{\begin{lem}}
\nc{\enlemma}{\end{lem}}
\nc{\Prop}{\begin{prop}}
\nc{\enprop}{\end{prop}}
\nc{\Th}{\begin{thm}}
\nc{\entheorem}{\end{thm}}
\nc{\Cor}{\begin{co}}
\nc{\encor}{\end{co}}

\ifx\pdfoutput\undefined
  \DeclareGraphicsExtensions{.pstex, .eps}
\else
  \ifx\pdfoutput\relax
    \DeclareGraphicsExtensions{.pstex, .eps}
  \else
    \ifnum\pdfoutput>0
      \DeclareGraphicsExtensions{.pdf}
    \else
      \DeclareGraphicsExtensions{.pstex, .eps}
    \fi
  \fi
\fi

\newcommand{\BB}{\mathbb{B}}
\newcommand{\ZZ}{\mathbb{Z}}

\newcommand{\QQ}{\mathbb{Q}}

\newcommand{\g}{\mathfrak{g}}
\newcommand{\h}{\mathfrak{h}}

\nc{\re}{\mathrm{re}}
\nc{\im}{\mathrm{im}}
\nc{\be}{\begin{enumerate}}
\nc{\ee}{\end{enumerate}}
\newcommand{\bni}{\be[{\rm(i)}]}
\newcommand{\bna}{\be[{\rm(a)}]}
\nc{\bnum}{\bni}

\nc{\seteq}{\mathbin{:=}}
\nc{\cl}{\colon}
\nc{\ol}{\overline}
\nc{\ro}{{\rm(}}
\nc{\rf}{{\rm)}}
\nc{\Z}{\mathbb{Z}}
\nc{\el}{\ell}

\nc{\cor}{\mathbf{k}}
\nc{\eq}{\begin{eqnarray}}
\nc{\eneq}{\end{eqnarray}}
\nc{\eqn}{\begin{eqnarray*}}
\nc{\eneqn}{\end{eqnarray*}}
\newcommand{\soplus}{\mathop{\mbox{\normalsize$\bigoplus$}}\limits}
\nc{\Proof}{\begin{proof}}
\nc{\QED}{\end{proof}}
\nc{\hs}{\hspace*}
\nc{\noi}{\noindent}
\nc{\B}{\mathbf{B}}
\nc{\f}{\mspace{1mu}\mathbf{f}\mspace{1mu}}
\nc{\e}{\mspace{1mu}\mathbf{e}\mspace{1mu}}
\nc{\wt}{\mathrm{wt}}
\nc{\la}{\lambda}
\renewcommand{\L}{\mathrm{L}}
\nc{\tp}{\mathrm{top}}
\nc{\bL}{{\mspace{1mu}\mathfrak{l}\mspace{.0mu}}}
\nc{\ba}{\begin{array}}
\nc{\ea}{\end{array}}
\nc{\ib}{{\mspace{1mu}\mathbf{i}\mspace{.0mu}}}
\nc{\Le}{\preceq}
\nc{\Lt}{\prec}
\nc{\epito}{\twoheadrightarrow}
\newcommand{\isoto}[1][]{\mathop{\xrightarrow%
[{\raisebox{.3ex}[0ex][.3ex]{$\scriptstyle{#1}$}}]%
{{\raisebox{-.6ex}[0ex][-.6ex]{$\mspace{2mu}\sim\mspace{2mu}$}}}}}
\newcommand{\monoto}{\rightarrowtail}
\nc{\bl}{\bigl(}
\nc{\br}{\bigl)}
\nc{\del}{\ell^\vee}
\nc{\E}{\mathbf{E}}
\nc{\Hom}{\mathrm{Hom}}
\nc{\lan}{\langle}
\nc{\ran}{\rangle}
\newcommand{\scbul}{{\,\raise.4ex\hbox{$\scriptscriptstyle\bullet$}\,}}
\nc{\F}{\mathbf{F}}
\renewcommand{\set}[2]{\left\{#1 \mid #2\right\}}
\nc{\Q}{\QQ}
\nc{\vs}{\vspace*}

\title{Dual Perfect Bases and dual perfect graphs}

\author{ Byeong Hoon Kahng$^{1}$ }
\address{Department of Mathematical Sciences, Seoul National University, 599 Gwanak-Ro, Seoul 151-747, Korea }
\email{peter546@snu.ac.kr}

\author{Seok-Jin Kang$^{2}$}
\address{Department of Mathematical Sciences and Research Institute of Mathematics, Seoul National University, 599 Gwanak-Ro, Seoul 151-747, Korea}
\email{sjkang@snu.ac.kr}

\author{Masaki Kashiwara$^{3}$}
\address{Research Institute for Mathematical Sciences, Kyoto University, Kyoto 606-8502, Japan, and Department of Mathematical Sciences, Seoul National University, 599 Gwanak-Ro, Seoul 151-747, Korea}
\email{masaki@kurims.kyoto-u.ac.jp}

\author{Uhi Rinn Suh$^{1}$}
\address{Research Institute of Mathematics, Seoul National University, 599 Gwanak-Ro, Seoul 151-747, Korea}
\email{uhrisu1@snu.ac.kr}

\thanks{$^{1}$This work was supported by NRF Grant \# 2014-021261}
\thanks{$^{2}$This work was supported by NRF Grant \# 2014-021261 and by NRF Grant \# 2013-055408}
\thanks{$^{3}$This work was partially supported by Grant-in-Aid for
Scientific Research (B) 22340005, Japan Society for the Promotion of
Science.}

\dedicatory{Dedicated to Professor Boris Feigin on the occasion of
his sixtieth birthday}

\date{May 8, 2014}

\begin{document}

\begin{abstract}
We introduce the notion of dual perfect bases and dual perfect
graphs. We show that every integrable highest weight module
$V_q(\lambda)$ over a quantum generalized Kac-Moody algebra
$U_{q}(\g)$ has a dual perfect basis  and its dual perfect graph is
isomorphic to the crystal $B(\lambda)$. We also show that the
negative half $U_{q}^{-}(\g)$ has a dual perfect basis whose dual
perfect graph is isomorphic to the crystal $B(\infty)$. More
generally, we prove that all the dual perfect graphs of a given dual
perfect space are isomorphic as abstract crystals. Finally, we show
that the isomorphism classes of finitely generated graded projective
indecomposable modules over a Khovanov-Lauda-Rouquier algebra and
its cyclotomic quotients form  dual perfect bases for their
Grothendieck groups.
\end{abstract}

\maketitle


\setcounter{tocdepth}{-1}

\pagestyle{plain}

\section*{Introduction} \label{Sec:Intro}

In \cite{BK}, Berenstein and Kazhdan introduced the notion of {\em
perfect bases} for integrable highest weight modules over Kac-Moody
algebras. Using the properties of perfect bases, they obtained
Kashiwara's crystal structure without taking quantum deformation and
crystal limits.

Their work was extended by  Kang, Oh and Park to the integrable
highest weight modules $V_q(\lambda)$  ($\lambda\in P^+$) over a
quantum generalized Kac-Moody algebra $U_q(\g)$ and to the negative
half $U_q^-(\g)$ \cite{KOP1, KOP2}. It was shown that the upper
global bases (or dual canonical bases) ${\mathbb B}(\lambda)$ and
${\mathbb B}(\infty)$ are perfect bases of $V_{q}(\lambda)$ and
$U_q^{-}(\g)$, respectively. They also showed that all the crystals
arising from perfect bases of $V_q(\lambda)$ and $U_q^-(\g)$ are
isomorphic to the crystals $B(\lambda)$ and $B(\infty)$,
respectively.

The perfect basis theory plays an important role in the
categorification of quantum generalized Kac-Moody algebras. To be
more precise, let ${\mathbb A}=\Z[q, q^{-1}]$ and let $U_{\mathbb
A}^{-}(\g)$ be the integral form of $U_q^{-}(\g)$. Let $R$ be the
Khovanov-Lauda-Rouquier algebra associated with the Borcherds-Cartan
datum for  $U_q(\g)$. We denote by $\mathrm{Rep}(R)$ the category of
finite-dimensional graded $R$-modules and let $[\mathrm{Rep}(R)]$
denote its Grothendieck group. Then it was proved that
$[\mathrm{Rep}(R)]$ is isomorphic to $U_{\mathbb A}^{-}(\g)^{\vee}$,
the dual of $U_{\mathbb A}^{-}(\g)$ with respect to a non-degenerate
symmetric bilinear form on $U_{q}^{-}(\g)$ \cite{KOP2, KL1, KL2, R}.
Moreover, in \cite{KOP2} (see also \cite{LV}), it was shown that the
isomorphism classes of finite-dimensional graded irreducible
$R$-modules form a perfect basis ${\mathbb B}$ of 
$\QQ(q)\otimes_{\ZZ[q,q^{-1}]} [\mathrm{Rep}(R)]$. Thus ${\mathbb
B}$ has a crystal structure which is  isomorphic to $B(\infty)$.

Similarly, the cyclotomic Khovanov-Lauda-Rouquier algebra
$R^{\lambda}$ $(\lambda\in P^+)$ provides a categorification of
$V_q(\lambda)$ in the sense that $[\mathrm{Rep}(R^{\lambda})]$ is
isomorphic to $V_{\mathbb A}(\lambda)^{\vee}$, the dual of the
integral form $V_{\mathbb A}(\lambda)$ of $V_q(\lambda)$ (\cite{KK,
KKO}). As in the case with $U_{q}^{-}(\g)$, the isomorphism classes
of finite-dimensional graded irreducible  $R^\lambda$-modules form a
perfect basis  ${\mathbb B}^{\lambda}$ of
$\QQ(q)\otimes_{\ZZ[q,q^{-1}]}[\mathrm{Rep}(R^{\lambda})]$
and ${\mathbb B}^{\lambda}$ has a crystal structure  which is
isomorphic to $B(\lambda)$.

On the other hand, let $\mathrm{Proj}(R)$ (respectively,
$\mathrm{Proj}(R^{\lambda})$) be the the category of finitely
generated graded projective $R$-modules (respectively,
$R^{\lambda}$-modules). Then they also provide a categorification of
$U_q^{-}(\g)$ and $V_{q}(\lambda)$. That is, in \cite{KK, KKO, KL1,
KL2, R}, it was shown that $[\mathrm{Proj}(R)]$ (respectively,
$[\mathrm{Proj}(R^{\lambda})]$) is isomorphic to $U_{\mathbb
A}^{-}(\g)$ (respectively, $V_{\mathbb A}(\lambda)$). Note that the
isomorphism classes of finitely generated graded projective
indecomposable modules form a basis of $[\mathrm{Proj}(R)]$ and
$[\mathrm{Proj}(R^{\lambda})]$, respectively. To describe their
properties, we need the dual notion of perfect bases.

When the Cartan datum is symmetric, the isomorphism classes of
finite-dimensional graded irreducible modules correspond to
Kashiwara's upper global basis and the isomorphism classes of
finitely generated graded projective indecomposable modules
correspond to Kashiwara's lower global basis(or Lusztig's canonical
basis) under the categorification (\cite{KKP,R2,VV}). However, when
the Cartan datum is not symmetric, the above statement is not true
in general.

Thus it is an interesting problem to characterize the bases of
$U_{q}^{-}(\g)$ and $V_q(\lambda)$ that correspond to the
isomorphism classes of finite-dimensional graded irreducible modules
and finitely generated graded projective indecomposable modules over
$R$ and $R^{\lambda}$.

In this paper, as the first step toward this direction, we introduce
the notion of {\em dual perfect bases} and {\em dual perfect
graphs}. The typical examples of dual perfect bases are the lower
global bases ${\mathbf B}(\infty)$ and ${\mathbf B}(\lambda)$ of
$U_q^{-}(\g)$ and $V_q(\lambda)$, receptively. It is straightforward
to verify that that their dual perfect graphs are isomorphic to the
crystals $B(\infty)$ and $B(\lambda)$.

More generally, we show that all the dual perfect graphs of a given
dual perfect space are isomorphic as abstract crystals. Thus every
dual perfect graph of $U_q^{-}(\g)$ is isomorphic to $B(\infty)$ and
the same statement holds for $V_{q}(\lambda)$.

 Finally, we show that the dual basis of a perfect basis is a
dual perfect basis.  Therefore, the isomorphism classes of
finitely generated graded projective indecomposable modules form a
dual perfect basis of  $\QQ(q) \otimes_{\ZZ[q,
q^{-1}]}[\mathrm{Proj}(R)] $ and  $\QQ(q)
\otimes_{\ZZ[q, q^{-1}]}[\mathrm{Proj}(R^{\lambda})]$,
respectively.

\vskip 3mm

{\bf Acknowledgments.} \ The second author would like to express
his sincere gratitude to RIMS, Kyoto University for their warm
hospitality during his visit in March 2014.

\vskip 5mm

\section{Quantum generalized Kac-Moody algebras} \label{Sec:GKM}

We first recall the basic theory of quantum generalized Kac-Moody
algebras. Let $I$ be an index set. An integral square
matrix $A=(a_{ij})_{i,j\in I}$ is called a {\em symmetrizable
Borcherds-Cartan matrix} if (i) $a_{ii}=2$ or $a_{ii} \le 0$ for all
$i \in I$,  (ii) $a_{ij} \le 0$ for $i \neq j$,  (iii) $a_{ij}=0$ if
and only if $a_{ji}=0$, (iv) there is a diagonal matrix
$D=\mathrm{diag}(s_i\in{\ZZ}_{>0})_{i\in I}$ such that $DA$ is
symmetric.

An index $i$ is {\em real} if $a_{ii}=2$ and is {\em imaginary} if
$a_{ii} \le 0$. We write $I^{\mathrm{re}} \seteq\{ i \in I \mid a_{ii}=2
\}$ and $I^{\mathrm{im}}\seteq \{ i \in I \mid a_{ii} \le 0 \}.$ In this
paper, we assume that $$a_{ii} \in 2\ZZ\quad\text{for all $i \in I$.}$$

A quadruple $(A, P, \Pi, \Pi^\vee)$ consisting of

\vs{.5ex}
\hs{4ex}\parbox{75ex}{
\begin{enumerate}[(a)]

\item a symmetrizable Borcherds-Cartan matrix $A=(a_{ij})_{i,j \in I}$,

\item a free abelian group $P$, which is called the {\em weight lattice},

\item the set of simple roots $\Pi=\{\alpha_i \in P \mid i\in I\},$

\item the set of simple coroots
$\Pi^\vee=\{h_i \in P^\vee \mid i\in I\}\subset P^\vee\seteq \mathrm{Hom}(P, \ZZ)$
\end{enumerate}}

\vs{.5ex}
is called a {\em Borcherds-Cartan datum} if it satisfies the
following properties:

\vs{.5ex}
\hs{1ex}\parbox{75ex}{
\begin{enumerate}
\item  $\left<h_i, a_j\right>=a_{ij}$ for $i,j \in I,$
\item  for any $i,j\in I$, there is $\Lambda_i \in P$ such that $\left< h_j, \Lambda_i\right> =\delta_{ij}$,
\item $\Pi$ is a linearly independent set.
\end{enumerate}
}

\vs{.5ex}
The subset $P^+\seteq\{\lambda \in P\,\mid \,\lambda(h_i)\in \ZZ_{\geq 0},
i \in I\}\subset P$ is called the set of dominant integral weights.
We denote by $Q\seteq\bigoplus_{i\in I} \ZZ \alpha_i$ the root lattice
and denote by $Q^+=\sum_{i\in I} \ZZ_{\geq 0} \alpha_i$ the positive
root lattice. We also call $\h\seteq\QQ \otimes_{\ZZ}P^{\vee}$ the
Cartan subalgebra.

\vskip 3mm We use the notation
 \begin{equation}
[n]_i\seteq \frac{q_i^n-q_i^{-n}}{q_i-q_i^{-1}},\ \
[m]_i !\seteq [m]_i[m-1]_i \cdots [1]_i, \ \
\left[ \begin{matrix} m_1 \\ m_2 \end{matrix} \right]_i\seteq \frac{[m_1]_i !}{[m_2]_i ![m_1-m_2]_i !},
\end{equation}
where $q_i= q^{s_i}$ for $i\in I$ and $[0]_q!\seteq1$.

\begin{defn}

The {\em quantum generalized Kac-Moody algebra} $U_q(\g)$
associated with a Borcherds-Cartan datum $(A, P,  \Pi, \Pi^{\vee})$
is the associative algebra over $\QQ(q)$ with unity generated by
$e_i, f_i$ $(i \in I)$ and $q^h$ $(h \in P^\vee)$ subject to the
following defining relations:
\begin{enumerate}
\item $q^0=1$,\,  $q^hq^{h'}=q^{h+h'},$
\item $q^he_iq^{-h}=q^{\alpha_i(h)}e_i$,\, $q^hf_iq^{-h}=q^{-\alpha_i(h)}f_i,$
\item $e_if_j-f_je_i=\delta_{ij}\dfrac{K_i-K_i^{-1}}{q_i-q_{-1}},$ where  $K_i=q^{s_ih_i},$
\item $\displaystyle \sum^{1-a_{ij}}_{k=0} (-1)^k \left[ \begin{matrix}1-a_{ij} \\ k \end{matrix} \right]_ie_i^{1-a_{ij}-k}e_je_i^k=0$ if $a_{ii}=2$ and $i\neq j,$
\item $\displaystyle \sum_{k=0}^{1-a_{ij}} (-1)^k \left[ \begin{matrix} 1-a_{ij} \\ k \end{matrix} \right]_if_i^{1-a_{ij}-k}f_jf_i^k=0 $ if $a_{ii}=2$ and $i\neq j,$
\item $e_ie_j-e_je_i=0$,\ $f_if_j-f_jf_i=0$ if $a_{ij}=0.$
\end{enumerate}
\end{defn}

For $k\in \mathbb{Z}_{>0}$, let
\begin{align*}
e_i^{(k)}=
 \begin{cases}
 \dfrac{e_i^k}{[k]_i!} \ & \text{if} \  i \in I^{\mathrm{re}}, \\
 e_i^k \ & \text{if} \ i \in I^{\mathrm{im}},
\end{cases} \qquad
f_i^{(k)}=
 \begin{cases}
 \dfrac{f_i^k}{[k]_i!} \ & \text{if} \  i \in I^{\mathrm{re}}, \\
 f_i^k \ & \text{if} \ i \in I^{\mathrm{im}}.
\end{cases}
\end{align*}

Let $U_q^-(\g)$ and $U_q^+(\g)$ be subalgebras of $U_q(\g)$
generated by $f_i$'s  and $e_i$'s $(i\in I)$, respectively, and let
$U^0_q$ be the  subalgebra of $U_q(\g)$ generated by $q^h$'s $(h \in
P^\vee)$. Then the algebra $U_q(\g)$ has the triangular
decomposition:
\[ U_q(\g)= U_q^-(\g)\otimes U_q^0(\g)\otimes U_q^+(\g).\]

Fix $i \in I$. For any $P\in U_q^-(\g)$, there exist unique elements $Q,R \in U_q^-(\g)$ such that
\begin{equation*}
e_iP-Pe_i=\frac{K_iQ-K_i^{-1}R}{q_i-q_i^{-1}}.
\end{equation*}
We define the endomorphisms $e'_i,e''_i\cl U_q^-(\g) \to U_q^-(\g)$ by
\begin{equation*}
e'_i(P)=R, \ \ e''_i(P)=Q.
\end{equation*}
Then we have the following commutation relations:
\begin{equation*} \label{e'}
e'_if_j=\delta_{ij}+q_i^{-a_{ij}}f_je'_i \ \ \text{for} \ i, j \in
I.
\end{equation*}

As we can see in \cite{JKK, K2}, there exists a unique
non-degenerate symmetric bilinear form $(\  , \ )$ on $U_q^-(\g)$
satisfying
\begin{equation} \label{Eqn:blf-infty}
(1, 1) =1, \quad  (f_i P , Q)=(P,e_i' Q) \quad \text{for all} \ \ P, Q
\in U^-_q(\g).
 \end{equation}

\vskip 3mm

We now turn to the representation theory of $U_q(\g)$.
 A $U_{q}(\g)$-module $V$ is called a {\em highest weight module}
with {\em highest weight} $\lambda$ if
 there exists a nonzero vector $v_{\lambda} \in V$ such that
\[(1)\ e_i v_\lambda=0 \text{ for } i\in I, \ \ (2)\ q^h v_\lambda= q^{\lambda(h)} v_\lambda\ \text{ for } h \in P^{\vee},  \ \ (3)\ U_q(\g) v_\lambda=V.  \]
\noindent The vector $v_{\lambda}$ is called  a  highest weight
vector. For each $\lambda \in P$, there exists a unique irreducible
highest weight module $V_q(\lambda)$  up to an isomorphism.

Consider the anti-involution $\phi$ on $U_q(\g)$ given by
$$q^h \mapsto q^h, \quad e_i \mapsto f_i, \quad f_i \mapsto e_i \ \
\text{for} \ \ i \in I, \, h \in P^{\vee}.$$ Then there exists a
unique non-degenerate symmetric bilinear form $(\ , \ )$ on
$V_q(\lambda)$ given by
\begin{equation}\label{eq:bilinear form on V(lambda)}
(v_{\lambda}, v_{\lambda})=1, \quad (P u, v) = (u, \phi(P)v) \quad
\text{for all $P \in U_q(\g)$ and $u, v \in V_q(\lambda)$.}
\end{equation}

\vskip 3mm

\begin{defn}
The {\em category ${\mathcal O}_{\mathrm{int}}$} consists of
$U_q({\g})$-modules $V$ satisfying the following properties:
\be[{\rm(a)}]
\item  $V$ has a weight space decomposition,
i.e.,  $V= \displaystyle\bigoplus_{\mu \in P} V_\mu$,
where $$V_\mu=\{ v \in V \mid \text{$q^hv=q^{\mu(h)}v$ for all $h \in P^{\vee}$} \},$$
\item there exist a finite number of weights $\lambda_1,\cdots, \lambda_t \in P$ such that
\begin{align*}
\mathrm{wt}(V)\seteq\{ \mu \in P \,\mid\, V_\mu \neq 0\} \subset
\bigcup_{j=1}^t(\lambda_j-Q^+),
\end{align*}
\item if $a_{ii}=2$, then the action of $f_i$ on $V$ is locally nilpotent,
\item if $a_{ii}\leq0$, then $\mu(h_i)\in{\ZZ}_{\geq 0}$ for every $\mu \in \mathrm{wt}(V),$
\item if $a_{ii}\leq0$ and $\mu(h_i)=0$, then $f_i (V_\mu)=0,$
\item if $a_{ii}\leq0$ and $\mu(h_i)=-a_{ii}$, then $e_i(V_\mu)=0$.
\end {enumerate}

\end{defn}

 The following
proposition was proved in \cite{JKK}.

\begin{prop} [\cite{JKK}]\hfill
\bni
\item The category ${\mathcal O}_{\mathrm{int}}$ is semisimple.

\item If $\lambda \in P^{+}$, then $V_q(\lambda)$
is a simple object of ${\mathcal O}_{\mathrm{int}}.$

\item Every simple object of the category ${\mathcal
O}_{\mathrm{int}}$ has the form $V_q(\lambda)$ for some $\lambda \in
P^{+}$.
\end{enumerate}
\end{prop}

\vskip 5mm

\section{Lower crystal bases} \label{Sec:LCB}

Let $V$ be a $U_q(\g)$-module in the category
$\mathcal{O}_{\mathrm{int}}$. 
It is straightforward to verify
that every vector $v \in V$ can be  uniquely written  as
\[ v= \sum_{k\ge0}f_i^{(k)} v_k,  \]
 where  $v_k\in\ker e_i$ with $v_k=0$ for $k\gg0$.
We define the {\em Kashiwara operators} $\widetilde{e}_i$ and
$\widetilde{f}_i$ $(i \in I)$ by
\[ \widetilde{e}_i v= \sum_{k=1}^{N} f_i^{(k-1)} v_k, \quad \widetilde{f}_i v= \sum_{k=0}^{N} f_i^{(k+1)} v_k.\]

Let
\[ \mathbb{A}_0\seteq \{f/g \in \QQ(q) \mid f, g \in \QQ[q],\ g(0) \neq 0\}.  \]

\begin{defn}
A free $\mathbb{A}_0$-submodule $L$ of $V$ is called a {\em lower
crystal lattice} of $V$ if it satisfies
\begin{enumerate}
\item $\QQ(q)\otimes_{\mathbb{A}_0} L =V,$

\item $L= \bigoplus_{\lambda \in P} L_\lambda$, where $L_{\lambda} = L\cap V_\lambda$,

\item $\widetilde{e}_i L \subset L$, $\widetilde{f}_i L \subset L$
for all $i \in I$.
\end{enumerate}
\end{defn}

\begin{defn}
Let $L$ be a lower crystal lattice of $V$ and let $B$ be a
$\QQ$-basis of $L \big/qL$. A pair $(L, B)$ is called a
{\em lower crystal basis} of $V$ if it satisfies
\begin{enumerate}
\item $B= \bigsqcup_{\lambda \in P} B_\lambda,$ where $B_\lambda= B \cap (L_\lambda / qL_\lambda)$,
\item $\widetilde{e}_i B \subset B \cup \{0\}, \ \  \widetilde{f}_i B \subset B \cup \{0\}$,
\item   for any $b, b'\in B$ and $i\in I$, we have
$\widetilde{f}_i(b)=b'$ if and only if $ \widetilde{e}_i b'= b$. 
\end{enumerate}
\end{defn}

We define the $I$-colored arrows on $B$ by setting $b\xrightarrow{i} b'$
if and only if $\widetilde{f_i}(b)=b'$. The $I$-colored  oriented  graph $(B,
\rightarrow)$ thus defined is called the {\em crystal graph} or
simply the {\em crystal} of $V$.

It is known that every $U_q(\g)$-module in
$\mathcal{O}_{\mathrm{int}}$ has a lower crystal basis.

\begin{prop}[\cite{JKK, K2}] \label{CBlambda}

Let $V_q(\lambda)$ $(\lambda \in P^{+})$ be the irreducible highest
weight module in the category ${\mathcal O}_{\mathrm{int}}$ with
highest vector $v_{\lambda}$. Then $V(\lambda)$ has a unique lower
crystal basis $(L(\lambda), B(\lambda))$, where
\begin{equation}\label{Eqn:Cry_Basis}
\begin{aligned}
& L(\lambda)= \mathbb{A}_0\text{-submodule generated by }\{ \widetilde{f}_{i_1} \cdots \widetilde{f}_{i_k} v_\lambda \mid\, i_s \in I, k\ge 0\}, \\
& B(\lambda)=  \{ \widetilde{f}_{i_1} \cdots \widetilde{f}_{i_k}
v_\lambda\  \mathrm{mod}\ q L(\lambda)  \mid \, i_s\in I, k \ge 0\}
\setminus \{0\}\subset L(\lambda) / q L(\lambda) .
 \end{aligned}
 \end{equation}
\end{prop}

Similarly, using the operator $e_{i}'$ in place of $e_i$, we can
develop the lower crystal basis theory for $U_{q}^{-}(\g)$. In
particular, we have the following proposition.

\begin{prop}[\cite{JKK, K2}] \label{l(infty)}
 Let
\begin{equation*}
\begin{aligned}
& L(\infty)= \mathbb{A}_0\text{-submodule generated by }\{ \widetilde{f}_{i_1} \cdots \widetilde{f}_{i_k} \boldmath{1} \mid\, i_s\in I, k \ge 0\},\\
& B(\infty)=  \{ \widetilde{f}_{i_1} \cdots \widetilde{f}_{i_k}
\boldmath{1}\ \mathrm{mod}\ q L(\infty)\mid \,i_s \in I, k \ge 0\}
\subset L(\infty) / q L(\infty).
\end{aligned}
\end{equation*}
Then the pair $(L(\infty), B(\infty))$ is a unique lower crystal
basis of $U_q^-(\g)$.
\end{prop}

Motivated by the properties of lower crystal bases, we define the
notion of {\em abstract crystals} and {\em crystal morphisms} as
follows.

\begin{defn}
An {\em abstract crystal} is a set B together with the maps
$\mathrm{wt}\cl B \to P$, $\varphi_i$, $\varepsilon_i\cl B \to {\ZZ}
\cup \{-\infty \}$ and $\widetilde{e_i},\widetilde{f_i}$ : $B \to
B\cup\{0\}$ with the following properties:
\begin{enumerate}
\item $\varphi_i(b)=\varepsilon_i(b)+\left<h_i, \mathrm{wt}(b)\right>.$
\item $\mathrm{wt}(\widetilde{e_i}b)=\mathrm{wt}(b)+\alpha_i$, $\mathrm{wt}(\widetilde{f_i}b)=\mathrm{wt}(b)-\alpha_i.$
\item $b=\widetilde{e_i}b'$ if and only if $ \widetilde{f_i}b=b'$,
where $b, b'  \in B$, $i \in I$.
\item If $\varphi_i(b)=-\infty$, then
$\widetilde{e_i}b=\widetilde{f_i}b=0$.
\item If $b\in B$ and $\widetilde{e_i}b \in B$,  then \[ \varepsilon_i(\widetilde{e_i}b)=\begin{cases} \varepsilon_i(b)-1 \ & \text{if} \ i \in I^{\re}, \\
\varepsilon_i(b) \ & \text{if} \ i \in I^{\im},\end{cases}\\
\ \ \varphi_i(\widetilde{e_i}b)=\begin{cases} \varphi_i(b)+1  \ &
\text{if} \ i \in I^{\re},\\ \varphi_i(b)+a_{ii}  \ & \text{if} \ i
\in I^{\im}.\end{cases}\]
\item If $b\in B$ and  $\widetilde{f_i}b \in B$, then  \[ \varepsilon_i(\widetilde{f_i}b)=\begin{cases} \varepsilon_i(b)+1 \ & \text{if} \  i \in I^{\re}, \\
\varepsilon_i(b) \ & \text{if} \ i \in I^{\im},\end{cases}\\
\ \ \varphi_i(\widetilde{f_i}b)=\begin{cases} \varphi_i(b)-1  \ &
\text{if} \  i \in I^{\re},\\ \varphi_i(b)-a_{ii}  \ & \text{if} \ i
\in I^{\im}.\end{cases}\]
\end{enumerate}
\end{defn}

\begin{defn}
Let $B_1$ and $B_2$ be abstract crystals. A {\em crystal morphism}
between $B_1$ and $ B_2$ is a map $\widetilde{\psi}:B_1 \to
B_2\sqcup \{0\}$ with the following properties:
\begin{enumerate}
\item For $b\in B_1$, $\widetilde{\psi}(b) \in B_2$ and $i\in I$,
we have $\mathrm{wt}(\widetilde{\psi}(b))= \mathrm{wt}(b)$,
$\varepsilon_i(\widetilde{\psi}(b))= \varepsilon_i(b)$ and $\varphi_i({\widetilde{\psi}}(b))= \varphi_i(b).$
\item  Suppose $b,b'\in B_1$ and  $\widetilde{\psi}(b), \widetilde{\psi}(b') \in B_2$.
If $\widetilde{f}_i b=b'$, then $\widetilde{f}_i(\widetilde{\psi}(b))=\widetilde{\psi}(b')$ and
$\widetilde{e}_i(\widetilde{\psi}(b'))=\widetilde{\psi}(b)$\, for $i\in I.$
\end{enumerate}
\end{defn}

\vskip 5mm

\section{Lower global bases} \label{sec:global}

Let $V$ be a vector space over $\QQ(q)$ and let
$$\mathbb{A} \seteq \QQ[q, q^{-1}], \quad
\mathbb{A}_\infty := \{ f\in \QQ(q)\mid f \text{ is regular at }
q=\infty\}.$$ Let $L_{\mathbb{A}}$ (respectively, $L_{0}$ and
$L_{\infty}$) be a free $\mathbb{A}$-submodule (respectively, free
$\mathbb{A}_{0}$-submodule and free $\mathbb{A}_{\infty}$-submodule)
of $V$ such that
$$\QQ(q) \otimes_{\mathbb{A}} L_{\mathbb{A}} \cong \QQ(q)
\otimes_{\mathbb{A}_{0}} L_{0}  \cong \QQ(q)
\otimes_{\mathbb{A}_{\infty}} L_{\infty} \cong V.$$

\begin{defn} \hfill
\begin{enumerate}
\item The triple $( L_{\mathbb{A}}, L_0, L_\infty)$ is called a {\em
balanced triple} of $V$ if the canonical projection $\pi:
L_{\mathbb{A}} \cap L_{0} \cap L_{\infty} \rightarrow L_{0} \big/ q
L_{0}$ is an isomorphism.

\item  Assume that $( L_{\mathbb{A}}, L_0, L_\infty)$ is a balanced triple.
Let $B$ be a $\QQ$-basis of $L_{0} \big/ qL_{0}$ and let  $G\cl
L_{0} \big/ q L_{0} \rightarrow L_{\mathbb{A}} \cap L_{0} \cap
L_{\infty}$ be the inverse of $\pi$. Then the set $\mathbf{B}\seteq
\{G(b) \mid \, b \in B\}$ is called the {\em lower global basis} of
$V$ corresponding to $B$.
\end{enumerate}
\end{defn}

Consider the $\QQ$-algebra automorphism $-: U_q(\g)\to U_q(\g)$
given by
\[ \ol{q}=q^{-1},\ \ \overline{e_i} = e_i, \ \ \overline{f_i}= f_i, \ \ \overline{q^h}= q^{-h}. \]
Define an involution  $-$ of $V_q(\lambda)$ by
\[ \overline{P v_\lambda} = \overline{P} v_\lambda,\]
where $P \in U_q(\g)$ and $v_\lambda\in V_q(\lambda)$ is the highest
weight vector.

Let $(L(\lambda), B(\lambda))$ be the lower crystal basis of
$V_q(\lambda)$ given in Proposition \ref{CBlambda} and let
$U_{\mathbb{A}}^{-}(\g)$ be the $\mathbb{A}$-subalgebra of
$U_{q}^{-}(\g)$ generated by $f_{i}^{(n)}$ $(i \in I, n\ge 0)$. Set
$$V_{\mathbb A}(\lambda)\seteq U_{\mathbb{A}}^{-}(\g) v_{\lambda}.$$ Then
$(V_{\mathbb A}(\lambda), L(\lambda), \overline{L(\lambda)})$ is a
balanced triple of $V_q(\lambda)$. Since $B(\lambda)$ is a
$\QQ$-basis of $L(\lambda)\big/ q L(\lambda)$, we obtain the lower
global basis $\mathbf{B}(\lambda) = \{G(b) \mid \, b \in B(\lambda)
\}$ of $V_q(\lambda)$.

Similarly, $(U_{\mathbb{A}}^{-}(\g), L(\infty),
\overline{L(\infty)})$ is a balanced triple of $U_{q}^{-}(\g)$ and
we get the lower global basis $\mathbf{B}(\infty) =  \{G(b) \mid \,
b \in B(\infty)\}$ of $U_{q}^{-}(\g)$.

\vskip 3mm

The lower global bases satisfy the following properties.

\begin{prop} [\cite{HK, JKK, K2}] \hfill
\bnum
\item For any $b \in B(\lambda)$ with $\lambda \in P^{+}$, $G(b)$ is
a unique element in $V_{\mathbb A}(\lambda) \cap L(\lambda)$  such that
$$\overline{G(b)} = G(b), \quad G(b) \equiv b \quad \mathrm{mod} \ q
L(\lambda).$$

\item For any $b \in B(\infty)$, $G(b)$ is a unique element in
$U_{\mathbb{A}}^{-}(\g) \cap L(\infty)$ 
such that
$$\overline{G(b)} = G(b), \quad G(b) \equiv b \quad
\mathrm{mod} \ q L(\infty).$$
\ee
\end{prop}

\begin{prop}[\cite{K4}] \label{prop:lgb}  The lower global bases $\mathbf{B}(\lambda)$ and
$\mathbf{B}(\infty)$ satisfy the following properties.

\bnum

\item For any $n \in {\ZZ}_{\ge 0}$ and $b \in B(\lambda)$
\ro respectively, $b \in B(\infty)$\rf, the subset
$\{G(b)\mid\varepsilon_i(b) \geq n \}$ of $\mathbf{B}(\lambda)$ \ro
respectively, of $\mathbf{B}(\infty)$\rf  \ is an $\mathbb{A}$-basis
of $\sum_{k \ge n} f_{i}^{(k)} V_{\mathbb A}(\lambda)$ \ \ro
respectively, of  $\sum_{k \ge n} f_{i}^{(k)}
U_{\mathbb{A}}^{-}(\g)$\rf.

\item For any $i \in I$ and $b \in B(\lambda)$ \ro respectively, $b \in
B(\infty)$\rf, we have
\begin{equation*}
\begin{array}{l}
f_iG(b)=
\begin{cases}
[1+\varepsilon_i(b) ]_iG(\widetilde{f_i}b)+\sum_{b'}F^{i}_{b,b'}G(b') \ &
\text{if $i \in I^{\re}$,}\\
G(\widetilde{f_i}b) \ &\text{if $i \in I^{\im}$,}
\end{cases}\\
 \end{array}
 \end{equation*}
 where the sum ranges over $b'$ such that $\varepsilon_i(b')>1+\varepsilon_i(b)$ and $F^{i}_{b,b'}\in \mathbb{A}$.

\end{enumerate}

\end{prop}

\vskip 5mm

\section{Dual perfect bases} \label{Sec:Dual PB}

In this section, we introduce the notion of {\em dual perfect
bases}. Fix a Borcherds-Cartan datum $(A, P, \Pi, \Pi^{\vee})$ and
let $\mathbf{k}$ be a field.

\begin{defn}\label{def:pre-perfect}
Let $V = \bigoplus_{\mu \in P} V_{\mu}$ be a  $P$-graded
$\cor$-vector space and let $\{f_i \}_{i\in I}$ be a family of
endomorphisms of $V$. We say that $(V,\{f_i\}_{i\in I})$ is a {\em
pre-dual perfect space} if it satisfies the following conditions.

\begin{enumerate}
\item There exist finitely many elements $\lambda_1, \ldots,
\lambda_k \in P $ such that $\mathrm{wt}(V) \subset \bigcup_{j=1}^k (\lambda_j - Q^{+})$.
\item
$f_{i} (V_{\mu}) \subset V_{\mu -
\alpha_i}$ for any $i\in I$ and $\mu\in P$.
\end{enumerate}
\end{defn}

\vskip 3mm

For each $i \in I$ and
$v\in V\setminus\{0\}$, define $\el_{i}(v)$ to be the non-negative integer $n$ such
that $v \in f_i^nV\setminus f_i^{n+1}V$.

\begin{defn}\label{Def: DPB}

Let $(V, \{f_i\}_{i \in I})$ be a pre-dual perfect space.

\bnum
\item  A basis $\mathbf{B}$ of $V$ is called a {\em dual perfect
basis} if

\bna
\item $\mathbf{B}=\bigsqcup_{\mu \in P} {\mathbf B}_{\mu}$, where
${\mathbf B}_{\mu} = {\mathbf B} \cap V_{\mu}$,

\item For any $i \in I$,  there exists a map
$\f_i\cl \B\to \B\cup\{0\}$ such that
for any $b \in {\mathbf B}$, there exists $c\in \mathbf{k}^\times$
satisfying
\begin{equation*}
 f_i(b)-c\,\mathbf{f}_i(b)\in f_i^{\el_i(b)+2}V.
\end{equation*}

\item If ${\mathbf f}_{i}(b) = {\mathbf f}_{i}(b')\not=0$, then $b
= b'$.

\end{enumerate}

\item $V$ is called a {\em dual perfect space} if it has a dual
perfect basis.

\end{enumerate}

\end{defn}

\vskip 3mm

\begin{prop} \label{prop:dual basis}
Every $U_{q}(\g)$-module in ${\mathcal O}_{\mathrm{int}}$ has a dual
perfect basis.
\end{prop}
\begin{proof}
It suffices to show that every irreducible highest weight module
$V_q(\lambda)$ $(\lambda \in P^{+})$ has a dual perfect basis. Let
${\mathbf B}(\lambda)= \{G(b) \mid \, b \in B(\lambda)\}$ be the
lower global basis of $V_q(\lambda)$. Define
$${\mathbf f}_{i}\, G(b)\seteq G(\widetilde{f_i}b) \ \ \text{for} \ b \in
B(\lambda).$$ Then by Proposition \ref{prop:lgb}, ${\mathbf
B}(\lambda)$ is a dual  perfect basis.
\end{proof}

\vskip 3mm

Let $\mathbf{B}$ be a dual perfect  basis of a $P$-graded
$\cor$-vector space $V$.

\Lemma\label{lem:fund} Let $i\in I$.
\bnum
\item
For any $b\in \mathbf{B}$ and $n\in \Z_{\ge0}$, there exists
$c\in\cor^\times$
such that
\eq&&f_i^nb-c\,\mathbf{f}_i ^n(b)\in f_i^{n+\el_i(b)+1}V.
\label{eq:fn}\eneq
\item
For any $i\in I$ and $n\in\Z_{\ge0}$, we have
$$f_i^nV=\soplus_{b\in \f_i^n\B}\cor\, b.$$
\item For any $b\in \mathbf{B}$, we have
$$\el_i(b)=\max\{n\in\Z_{\ge0}\mid b\in \mathbf{f}_i^n\mathbf{B}\}.$$
\item  For any $b\in \mathbf{B}$ such that $\mathbf{f}_ib\not=0$,
we have $$\el_i(\mathbf{f}_ib)=\el_i(b)+1.$$
\item For any $n\in\Z_{\ge0}$,
the image of $\{b\in \B\mid\el_i(b)=n\}$
is a basis of $f_i^nV/f_i^{n+1}V$.
\ee \enlemma \Proof  (i)\
By the definition, we have
$\el_i(\f_ib)\ge\el_i(b)+1$ for any $b\in\B$ such that $\f_ib\in\B$.
Hence we have
$$
\text{$\el_i(\f_i^nb)\ge\el_i(b)+n$ for any $b\in\B$ and $n\in\Z_{\ge0}$
such that $\f_i^nb\in\B$.}
$$
We shall show (i) by induction on $n$.
If $n=0,1$, (i) is obvious.
Assume $n>1$.
Then the induction hypothesis implies
that $f_i^{n-1}b-c\f_i^{n-1}b\in f_i^{n+\el_i(b)}V$ for some $c\in \cor^\times$.
Therefore, we have $f_i^{n}b-cf_i\f_i^{n-1}b\subset f_i^{n+\el_i(b)+1}V$.
Hence, if $\f_i^{n-1}b=0$, then we obtain \eqref{eq:fn}.
If $\f_i^{n-1}b\in \B$, then we have
$f_i\f_i^{n-1}b-c'\,\f_i^{n}b\in f_i^{\el_i(\f_i^{n-1}b)+2}V\subset
f_i^{n-1+\el_i(b)+2}V$ for some $c'\in \cor^\times$.
Hence we obtain (i).

\smallskip
\noi (ii)  By (i), we have
$\f_i^nb\in f_i^nV$. Hence it is enough to show that
$f_i^nV\subset\soplus_{b\in \f_i^n\mathbf{B} }\cor\,b$.  We have
$$f_i^nV\subset\soplus_{b\in \f_i^n\mathbf{B} }\cor\,
b+f_i^{n+1}V,$$
which easily follows from (i).
 Then (ii) follows from
$(f_{i}^k V)_{\lambda} =0$ for any $\lambda\in P$ and $k\gg0$. 

(iii), (iv) and  (v) easily follow from (ii). \QED

Let $\mathbf{B}$ be a dual perfect  basis of a $P$-graded $\cor$-vector
space $V$. For $b \in {\mathbf B}$, we define
\begin{align*}
\mathbf{e}_ib=
\begin{cases}
b'  & \mbox{if $b'\in\mathbf B$ satisfies $\mathbf{f}_i \,b'=b$,} \\
0  & \mbox{if there exists no $b'\in\mathbf B$ such that
$\mathbf{f}_i \,b'=b$.}
\end{cases}
\end{align*}
We also define the maps $\varepsilon_i,
\varphi_i:\mathbf{B}\to \ZZ \cup \{- \infty \}$  by
\begin{align*}
& \varepsilon_i(b)=
\begin{cases}
\el_i(b) &\text {if $i \in I^{\mathrm{re}}$,} \\
0 \quad &\text{if $i \in I^{\mathrm{im}}$,}
\end{cases} \\
&\varphi_i(b)=\varepsilon_i(b)+\left<h_i,\mathrm{wt}(b) \right>,
\end{align*}
where the map $\mathrm{wt}: {\mathbf B} \rightarrow P$ is given by
$\mathrm{wt}(b) = \mu$ if $b \in {\mathbf B}_{\mu}$.

Then it is straightforward to verify that $(\mathbf{B},
\mathbf{e}_{i}, \mathbf{f}_{i}, \varepsilon_{i}, \varphi_{i},
\mathrm{wt})$ is an abstract crystal, which will be called the {\em
dual perfect graph}.

\begin{prop} \label{prop:dual graph}
Let ${\mathbf B}(\lambda)$ be the global basis of $V(\lambda)$
$(\lambda \in P^{+})$. Then ${\mathbf B}(\lambda)$ is isomorphic to
$B(\lambda)$ as an abstract crystal.
\end{prop}

\begin{proof}
Recall that $\mathbf{B}(\lambda)$ becomes a dual perfect basis by
defining $\mathbf{f}_{i} G(b) = G (\widetilde{f_i}b)$ for $b\in
B(\lambda)$. Hence for $b, b' \in B(\lambda)$, we have $\widetilde{f_{i}} b
= b'$ if and only if $G(b')=G(\widetilde{f_{i}}b) = \mathbf{f}_{i}
G(b)$, which proves our claim.
\end{proof}

\vskip 3mm

\begin{rem} Let $\g$ be the generalized Kac-Moody algebra associated
with a Borcherds-Cartan datum. By taking the classical limit in
Proposition \ref{prop:lgb}, it follows that every irreducible
highest weight $\g$-module $V(\lambda)$ $(\lambda \in P^{+})$ has a
dual perfect basis whose dual perfect graph is isomorphic to
$B(\lambda)$ as an abstract crystal.
\end{rem}

\vskip 3mm

\begin{prop} \label{prop:dual Binfty} \hfill
\bnum
\item The algebra $U_{q}^{-}(\g)$ has a dual perfect basis.

\item The lower global basis $\mathbf{B}(\infty)$ is isomorphic to
$B(\infty)$ as an abstract crystal.
\ee
\end{prop}

\begin{proof}
As in Proposition \ref{prop:dual basis} and Proposition
\ref{prop:dual graph}, the lower global basis $\mathbf{B}(\infty)$
satisfies our assertions.
\end{proof}

\vskip 5mm

\section{Properties of dual perfect bases}
 Let $\cor$ be a field, and
$\B$ be a dual perfect basis of $(V,\{f_i\}_{i\in I})$.

For $b\in \B$ and $i \in I$, we set
$\e_i^{\tp}(b)=\e_i^{\el_i(b)}b\in\B$. More generally, for a
sequence $\bold{i}=(i_1,\cdots, i_m) \in I^m$ ($m\ge1$), we set
$$\e_{\ib}^\tp b:=\e_{i_m}^\tp\cdots\e_{i_1}^\tp b.$$

We also use the notations
\eqn&&
\e_{\ib}^{\,\bL}\; b\seteq \e_{i_m}^{l_m}\cdots \e_{i_1}^{l_1}b
\in \B\sqcup\{0\},\\
&&\f_{\ib}^{\,\bL}\; b\seteq \f_{i_1}^{l_1}\cdots \f_{i_m}^{l_m}b
\in \B\sqcup\{0\} \eneqn for $\bL=(l_1,\ldots,l_m)\in\Z_{\ge0}^m$
and set
$$f_{\ib}^{\,\bL}\seteq f_{i_1}^{l_1}\cdots f_{i_m}^{l_m}.$$

We say that a sequence $\ib=(i_k)_{k\ge1}$ is {\em good} if
$\{k\in\Z_{>0}\mid i_k=i\}$ is an infinite  set  for any $i\in I$.
We say that a sequence $\bL=(l_k)_{k\ge 1}$ of non-negative integers
is {\em good} if $l_k=0$ for $k\gg0$. For a good sequence
$\ib=(i_k)_{k\ge1}$ in $I$ and a good sequence $\bL=(l_k)_{k\ge 1}$
of non-negative integers, we set
\begin{align*}
& V^{>\bL,\;\ib}=\sum_{k\ge1}f_{i_1}^{l_1}\cdots f_{i_{k-1}}^{l_{k-1}}
f_{i_k}^{1+l_k}V,\\
& V^{\ge\bL,\;\ib}=V^{>\bL,\;\ib}
+f_{i_1}^{l_1}\cdots f_{i_{m}}^{l_{m}}V=\sum_{k=1}^mf_{i_1}^{l_1}\cdots f_{i_{k-1}}^{l_{k-1}}
f_{i_k}^{1+l_k}V+f_{i_1}^{l_1}\cdots f_{i_{m}}^{l_{m}}V
\end{align*}
 for $m\gg0$.
\vskip 3mm

Let $\Lt$ be the lexicographic ordering on good sequences of
integers, namely, $\bL =(l_k)_{k\ge1}\Lt \bL'=(l'_k)_{k\ge1}$ if and
only if there exists $k\ge1$ such that $l_s=l'_s$ for any $s$ with
$1\le s<k$ and $l_k<l'_k$. Then for $\bL=(l_k)_{k\ge1}$ and
$\bL'=(l'_k)_{k\ge1}$, we have
\begin{align*}
& V^{\ge\bL,\;\ib}\subset V^{\ge\bL',\;\ib}\quad\text{if $\bL'\Le \bL$,}\\
& V^{>\bL,\;\ib}\subset V^{\ge\bL',\;\ib}\quad\text{if $\bL'\Lt \bL$.}
\end{align*}

For any $v \in V\setminus\{0\}$ and a good sequence
$\ib=(i_k)_{k\ge1}$ in $I$, we define $\L(\mathbf
i,v)=(l_k)_{k\ge1}$ to be the largest sequence $\bL=(l_k)_{k\ge1}$
such that $v\in V^{\ge \bL}$. Note that  such an $\L(\mathbf i,v)$ exists and
is a good
sequence (see Proposition~\ref{prop:main}~\eqref{L=L}). Hence $v\in V^{\ge \bL,\;\ib}$ if and only if $\bL\Le
\L(\ib,v)$.

Set $\B_H=\{b\in \B\mid \text{$\el_i(b)=0$ for  all  $i\in I$}\}$. For
a good sequence $\ib = (i_{k})_{k\ge 1}$, we set
$\e_{\ib}^\tp=\e_{i_m}^\tp\cdots \e_{i_1}^\tp b$ for $m\gg0$. Note
that it does not depend on $m\gg0$ and belongs to $\B_H$
(see Proposition~\ref{prop:main}~\eqref{L=L}).

\vskip 3mm

\Prop\label{prop:main}
Let $\mathbf {i} = (i_k)_{k\ge1}$
be a good sequence in $I$.
\bnum
\item For any $b\in \B$ and a sequence
 $\bL = (l_1,\ldots,l_m)$ of non-negative integers,
there exists $c\in\cor^\times$ such that
$$ f_{i_1}^{l_1}\cdots f_{i_m}^{l_m}b-c
\,\f_{i_1}^{l_1}\cdots \f_{i_m}^{l_m}b\in
\sum_{k=1}^mf_{i_1}^{l_1}\cdots f_{i_k}^{1+l_k}V.$$
\item
For each $b \in \B$, define a sequence $(b_{k})_{k \ge
0}$ by
$$b_{0}=b, \ \ b_{k} = \e_{i_k}^{\tp} b_{k-1} \quad  \text{for} \
 \  k \ge 1,$$
and let $(\L_{k})_{k \ge 1}$ be the sequence of non-negative
integers given by
$$\L_{k}= l_{i_k}(b_{k-1}) \quad \text{for} \ \ k\ge 1.$$

 Then we have
\bna
\item $(\L_k)_{k\ge1}$ is a good sequence,
\item $b_k\in B_H$ for $k\gg0$,
\item $\L(\mathbf i,b)= (\L_k)_{k\ge1}$.
\ee
\label{L=L}
\item For any good sequence $\bL=(l_k)_{k\ge1}$ of  non-negative
integers,
we have
\eq
&&V^{\ge\bL,\;\ib}=
\sum_{\substack{\{b\in\mathbf B\;\mid \;
\bL\Le\L(\mathbf i,b)\}}}\cor\,b,\label{eq:Ve}\\
&&V^{>\bL,\;\ib}=
\sum_{\substack{\{b\in\mathbf B\;\mid \;
\bL\Lt\L(\mathbf i,b)\}}}\cor\,b.\label{eq:Vg}
\eneq

\item  For any good sequence $\bL=(l_k)_{k\ge1}$ of  non-negative integers,
we have an injective map
\begin{equation}\label{eqn:e^l}
\e^\bL\cl \{b\in\B\mid\L(\ib,b)=\bL\}\monoto \B_H.
\end{equation}
\ee
\enprop
\Proof
 Let us first prove (i) by induction on $m$.
The $m=1$ case follows from Lemma~\ref{lem:fund}~(i).
Assume that $m>1$.
Set $b_1=\f_{i_m}^{l_m}b$.
Then applying the induction hypothesis to $b_1$, there exists $c\in\cor^\times$ such that
$$f_{i_1}^{l_1}\cdots f_{i_{m-1}}^{l_{m-1}}b_1-c
\,\f_{i_1}^{l_1}\cdots \f_{i_{m-1}}^{l_{m-1}}b_1 \in
\sum_{k=1}^{m-1}f_{i_1}^{l_1}\cdots f_{i_k}^{1+l_k}V.$$
By Lemma \ref{lem:fund} (i), there exists $c'\in\cor^\times$
such that $f_{i_m}^{l_m}b -c'b_1\in f_{i_m}^{1+l_m}V$. Hence we have
$$
\begin{aligned}
& f_{i_1}^{l_1} \cdots f_{i_m}^{l_m} b - c c' \f_{i_1}^{l_1} \cdots
\f_{i_m}^{l_m} b = f_{i_1}^{l_1} \cdots f_{i_{m-1}}^{l_{m-1}}
(f_{i_m}^{l_m} b - c' b_{1}) \\
& \qquad + c'(f_{i_1}^{l_1} \cdots f_{i_{m-1}}^{l_{m-1}} b_{1} - c
\f_{i_1} \cdots \f_{i_{m-1}}^{l_{m-1}} b_{1}) \in
\sum_{k=1}^mf_{i_1}^{l_1}\cdots f_{i_k}^{1+l_k}V.
\end{aligned}
$$

\medskip

 Next we shall show (ii) (a) and (ii) (b). Since
$\wt(b)-\sum_{k=1}^m\L_k\alpha_{i_k}\in \wt(V)$ for any $m$, we have
$\L_k=0$ for $k\gg0$. Hence $b_k$ does not depend on $k\gg0$.
 Thus $\el_{i_{k+1}}(b_k)=0$ for $k\gg0$ which implies
(ii) (b).

\medskip

To prove  (ii) (c)  and (iii),  let $\L=(\L_k)_{k\ge1}$ be
the sequence in (ii) which is uniquely determined for each $b\in \B$. For
$m\in\Z_{\ge0}$, set $\widetilde{\L}_m(\mathbf i,b)=(\L_{1}, \ldots,
\L_{m})$. We first observe that for any sequence $\bL=(l_1, \ldots,
l_m)$ with $\bL \Le \widetilde{\L}_m(\mathbf i,b)$, we have \eq
&&\text{$b\in \sum_{k=1}^{m-1}f_{i_1}^{l_1}\cdots
f_{i_{k-1}}^{l_{k-1}}f_{i_k}^{1+l_k}V +f_{i_1}^{l_1}\cdots
f_{i_m}^{l_m}V$\quad if $(l_1, \ldots,
l_{m})\Le\widetilde{\L}_m(\mathbf i,b)$}, \label{eq:bb} \eneq which
immediately follows from (i).

Now we shall show
\eq&&
\sum_{k=1}^{m-1}f_{i_1}^{l_1}\cdots f_{i_{k-1}}^{l_{k-1}}f_{i_k}^{1+l_k}V
+f_{i_1}^{l_1}\cdots f_{i_m}^{l_m}V=
\sum_{\substack{\{b\in\mathbf B\;\mid \;
(l_1, \ldots, l_{m})\Le\widetilde{\L}_m(\mathbf i,b)\}}}\cor\,b.
\label{eq:bbb}
\eneq
We have already seen that the right hand side is contained in the left hand side.
Let us show the converse inclusion by induction on $m$.
In order to see this, it is enough to show that
$$f_{i_1}^{l_1}\cdots f_{i_m}^{l_m}V\subset
\sum_{\substack{\{b\in\mathbf B\;\mid \;
(l_1, \ldots, l_{m})\Le\widetilde{\L}_m(\mathbf i,b)\}}}\cor\,b.
$$ 
Set  $\mathbf i'=(i_2,i_3,\ldots)$. 
Then the induction hypothesis implies
$$f_{i_2}^{l_2}\cdots f_{i_m}^{l_m}V\subset
\sum_{\substack{\{b\in\mathbf B\;\mid \;
(l_2, \ldots, l_{m})\Le\widetilde{\L}_{m-1}(\mathbf i',b)\}}}\cor\,b.
$$
Hence we have reduced \eqref{eq:bbb} to
$$f_{i_1}^{l_1}b_0\in
\sum_{\substack{\{b\in\mathbf B\;\mid \;
(l_1, \ldots, l_{m})\Le\widetilde{\L}_m(\mathbf i,b)\}}}\cor\,b
$$
if $(l_2, \ldots, l_{m})\Le\widetilde{\L}_{m-1}(\mathbf i',b_0)$, which
follows from the fact that $f_{i_1}^{l_1}b_0\in \cor \,\mathbf
f_{i_1}^{l_1}b_0+\sum_{\el_i(b)>l_1}\cor b$ and $(l_1, \ldots,
l_{m})\Le\widetilde{\L}_m(\mathbf i,\mathbf f_{i_1}^{l_1}b_0).$ Thus
the proof of \eqref{eq:bbb} is complete.

\medskip
Now,   (ii) (c)  follows from \eqref{eq:bbb}, and then \eqref{eq:Ve} is
nothing but \eqref{eq:bbb} for $m\gg0$. Equality \eqref{eq:Vg} follows easily
from \eqref{eq:Ve}.

\medskip
In order to prove (iv), observe that
$\mathbf{e}^\bL(b)=\mathbf{e}_{\mathbf{i}}^{\mathrm{top}}(b)$, where
$\mathbf{e}^\bL$ is the map defined in  (\ref{eqn:e^l}) and
$L(\mathbf{i}, b)=\bL.$ Since $\mathbf{e}^\bL$ has  a left inverse 
$\mathbf{f}_{\mathbf{i}}^\bL|_{\mathbf{B}_H}$, we conclude that
$\mathbf{e}^\bL$ is injective. \QED

The following corollary easily follows from the preceding
proposition.

\Cor \label{cor:main} Let $\ib=(i_k)_{k\ge1}$  be a good sequence
 in $I$ and let  $\bL=(l_k)_{k\ge1}$ be a good
sequence of non-negative integers. Denote by  $p_{\ib,\bL}\cl
V^{\ge\bL,\ib}\to V^{\ge\bL,\ib}/V^{>\bL,\ib}$ the canonical
projection and set $\B_{\ib,\bL}=\{b\in\B\mid\L(\ib,b)=\bL\}$.

 Then the  image
$p_{\ib,\bL}(\B_{\ib,\bL})$ is a basis of $V^{\ge
\bL,\,\ib}/V^{>\bL,\,\ib}$. Moreover, $\cor^\times
p_{\ib,\bL}(\B_{\ib,\bL})$ is equal to $\cor^\times
\bl p_{\ib,\bL}(f_\ib^{\bL}\B_H)\setminus\{0\}\br$. \encor

\noindent Here, for a subset $S$ of a $\cor$-vector space $V$, we
use the notation
$$\cor^\times S\seteq\{\cor^\times s\mid s\in S\}.$$

\medskip
Note that Proposition~\ref{prop:main} (iii) with $\bL=(0,0,\ldots)$
implies that
$$\sum_{i\in I}f_iV=\soplus_{b\in\B\setminus \B_H}
\cor b.$$
Hence we conclude
\Lemma \label{lem:main} Set $V_H\seteq
V/(\sum_{i\in I}f_iV)$ and let $p_H\cl V\epito V_H$ be the canonical
projection. Then $p_H\cl \B_H\to V_H$ is injective and $p_H(\B_H)$
is a basis of $V_H$. \enlemma

\vskip 5mm

\section{Uniqueness of dual perfect graphs} \label{Sec:uniqueness}
The purpose of this section is to prove that all the dual perfect
graphs of a given dual perfect space are isomorphic as abstract crystals.

\Th \label{Thm:uniqueness} Let $(V, \{f_i\}_{i\in I})$ be a dual
perfect space and let $\B$ and $\B'$ be its dual perfect bases.
Assume that $p_H(\B_H)=p_H(\B'_H)$.

Then there is a crystal isomorphism $\psi\cl \B\isoto \B'$ such that
$p_H(b)=p_H(\psi(b))$ for all $b\in\B_H$. Moreover, for any $b\in\B$
and a good sequence $\ib$ in $I$, we have $\L(\ib,
b)=\L(\ib,\psi(b))$ and $p_H(\e_{\ib}^\tp b)=p_H(\e_{\ib}^\tp
\psi(b)) \in p_H(\B_H)$. \entheorem

\Proof For a good sequence $\ib=(i_k)_{k\ge1}$ in $I$ and a good
sequence $\bL=(l_k)_{k\ge1}$, let $p_{\ib,\bL}\cl V^{\ge\bL,\ib}\to
V^{\ge\bL,\ib}/V^{>\bL,\ib}$ be the canonical projection. Set
$\B_{\ib,\bL}=\{b\in\B\mid\L(\ib,b)=\bL\}$ and define
$\B'_{\ib,\bL}$ in a similar manner.  Then by Corollary
\ref{cor:main} and Lemma \ref{lem:main}, we have
$$\cor^\times p_{\ib,\bL}(\B_{\ib,\bL})=\cor^\times p_{\ib,\bL}(\B'_{\ib,\bL})$$
and both $p_{\ib,\bL}(\B_{\ib,\bL})$ and
$p_{\ib,\bL}(\B'_{\ib,\bL})$ are bases of
$V^{\ge\bL,\ib}/V^{>\bL,\ib}$.
Hence for any $b\in \B$,  there exists $b'\in \B'$ such that \eq
&&\text{$\L(\ib,b)=\L(\ib,b')$ and $b-cb'\in V^{>\L(\ib,b),\ib}$ for
some $c\in \cor^\times$.} \label{eq:6.1} \eneq

To prove our claim, it is enough to show that for another choice of
a good sequence $\ib'$, \eqref{eq:6.1} holds with the same $b'$. Set
$v=b-cb'$. Since $v\in V^{>\L(\ib,b),\ib}$, $v$ is a linear
combination of $\B\setminus\{b\}$. Set $\bL=\L(\ib',b)$ and
$\bL'=\L(\ib',b')$. Then $b\in V^{\ge\bL,\ib'}$ and $b'\in
V^{\ge\bL',\ib'}$. Assume that $\bL\Lt \bL'$. Since $b'\in
V^{>\bL,\ib'}$, we have $v-b\in V^{>\bL,\ib'}$. Hence $v-b$ is a
linear combination of $\B\setminus\{b\}$, which is a contradiction to
the fact that $v$ is a linear combination of $\B\setminus\{b\}$.
Hence we conclude that $\bL'\Le\bL$. Similarly, we have
$\bL\Le\bL'$. Hence we obtain $\bL=\bL'$. It follows that both
$p_{\ib',\bL}(b)$ and $p_{\ib',\bL}(b')$ belong to $\cor^\times
p_{\ib',\bL}(\B_{\ib,\bL})$. If $\cor^\times
p_{\ib',\bL}(b)\not=\cor^\times p_{\ib',\bL}(b')$, then $v-b$ is a
linear combination of $\B\setminus\{b\}$, which is a contradiction. Hence
$\cor^\times p_{\ib',\bL}(b)=\cor^\times
p_{\ib',\bL}(b')$ and our assertion follows. \QED

\vskip 3mm

\begin{co} \label{cor:dpg}
Let $U_{q}(\g)$ be the quantum generalized Kac-Moody algebra
associated with a Borcherds-Cartan datum and let $V_q(\lambda)$ be
the irreducible highest weight $U_{q}(\g)$-module with $\lambda \in
P^{+}$. Then the following statements hold.

\bnum
\item Every dual perfect graph of $V_q(\lambda)$ is isomorphic to
$B(\lambda)$ as an abstract crystal.

\item Every dual perfect graph of $U_{q}^{-}(\g)$ is isomorphic to
$B(\infty)$ as an abstract crystal.

\ee
\end{co}

\vskip 3mm

\begin{rem} Theorem \ref{Thm:uniqueness} also shows that every
dual perfect graph of $V(\lambda)$ over a generalized Kac-Moody
algebra $\g$ is isomorphic to $B(\lambda)$ as an abstract crystal.
\end{rem}

\vskip 5mm

\section{perfect bases and dual perfect bases}

Now we will prove that the isomorphism classes of finitely generated
graded projective modules over $R$ and $R^{\lambda}$ form a dual
perfect basis of $\Q(q)\otimes_{\Q[q,q^{-1}]}[\mathrm{Proj}(R)]$ and
$\Q(q)\otimes_{\Q[q,q^{-1}]}[\mathrm{Proj}(R^{\lambda})]$, respectively. We first recall
the definition of {\em perfect basis}.

Let $V=\bigoplus_{\mu \in P} V_{\mu}$ be a $P$-graded  $\cor$-vector space 
with a family of endomorphisms $\{e_{i}\}_{i \in I}$ satisfying the
following conditions.
\begin{enumerate}[(i)]
\item There exist finitely many elements $\lambda_1, \ldots,
\lambda_k \in P $ such that $\mathrm{wt}(V) \subset \bigcup_{j=1}^k
(\lambda_j - Q^{+})$.
\item
$e_{i} (V_{\mu}) \subset V_{\mu + \alpha_i}$ for any $i\in I$ and
$\mu\in P$.
\end{enumerate}

 Set $\del_i(v):= \max\{n\in \Z_{\geq 0}\mid e_i^{n}v\not=0\}
=\min\{n\in \Z_{\geq 0}\mid e_i^{n+1}v=0\}$
for $v\in V\setminus\{0\}$. 

\vskip 3mm

\begin{defn} \label{Def:PB}
A basis $\mathbb{B}$ of $V$ is said to be {\em perfect} if
\begin{enumerate}
\item $\mathbb{B}=\bigoplus_{\mu \in P} \mathbb{B}_\mu$, where $\mathbb{B}_\mu:= \mathbb{B} \cap V_\mu,$
\item  for any $i\in I$,
there exists a map $\mathbf{E}_i\cl \BB\to \BB\cup\{0\}$
such that for any $b \in \mathbb{B}$, we have
\bna
\item if $\del_i(b)=0$, then $\mathbf{E}_i(b)=0$,
\item if $\del_i(b)>0$, then $\E_i (b)\in\BB$ and
$$e_i(b)- c \mathbf{E}_i(b)\in \ker e_i^{\,\del_i(b)-1}
\quad\text{for some $c \in \cor^\times$,}$$
\ee
\item if $\mathbf{E}_i (b)= \mathbf{E}_i (b') \neq 0$ for $b, b' \in \mathbb{B}$, then $b=b'.$
\end{enumerate}
\end{defn}
It was shown in \cite{KOP1} and \cite{KOP2} that every
$U_q(\g)$-module in the category ${\mathcal O}_{\mathrm{int}}$ and the
negative half $U_q^{-}(\g)$ have perfect bases.

\vskip 3mm

For a perfect basis $\BB$, we define a map
$\mathbf{F}_i\cl \mathbb{B}\to \mathbb{B} \cup \{0\}$
by
\begin{equation*}\mathbf{F}_i (b) = \left\{
\begin{array}{ll}
b' & \ \  \text{if $\mathbf{E}_i  (b')=b$,} \\
0  & \ \ \text{if $b\notin\E_i\BB$.}
\end{array}\right.
\end{equation*}
Let us recall the following lemma.
\Lemma[\cite{BK}]
Let $\BB$ be a perfect basis of $(V,\{e_i\}_{i\in I})$.
Then 
\bnum
\item for any $n\in\Z_{\ge0}$, we have
$$\ker e_i^n=\soplus_{b\in \BB,\;\del_i(b)<n}\cor b.$$
\item $\del_i(\E_ib)=\del_i(b)-1$ for any $b\in\BB$ such that
$\del_i(b)>0$.
\item for $b\in\BB$ and $i\in I$, we have
$b\in\E_i\BB$ if and only if $b\in e_iV+\ker e_i^{\del_i(b)}$.
\item for any $n\in\Z_{\ge0}$ and $i\in I$,
the image of $\{b\in \BB\mid \del_i(b)=n\}$
is a basis of $\ker e_i^{n+1}/\ker e_i^{n}$.
\ee
\enlemma

Let $(V,\{f_i\}_{i\in I})$ be a pre-dual perfect space
such that $\dim V_\la<\infty$ for any $\la\in P$.
We set
\eq
&&(V^\vee)_\la=\Hom_\cor(V_\la, \cor)
\quad\text{for any $\la\in P$, and}\quad
V^\vee=\soplus_{\la\in P}(V^\vee)_\la.
\eneq
Let $\lan\scbul,\scbul\ran: V\times V^\vee \to \cor$
be the canonical pairing.
We define $e_i\cl V^\vee \to V^\vee$ as the transpose of $f_i$,
so that we have
$\lan u, e_iv\ran=\lan f_iu,v\ran$ for any $u\in V$ and $v\in V^\vee$.

\Prop
Let $(V,\{f_i\}_{i\in I})$ be a  a pre-dual perfect space
such that $\dim V_\la<\infty$ for any $\la\in P$.
Let $\B$ be a basis of $V$ and let
$\B^\vee\subset V^\vee$ be the dual basis of $\B$.
Then
\bnum
\item
$\B$ is a dual perfect basis
if and only if $\B^\vee$ is a perfect basis.
\item
Assume that $\B$ is a dual perfect basis.
Denoting the canonical isomorphism
$\B\isoto \B^\vee$ by $ b\mapsto b^\vee$ ,
we have
\eqn
\el_i(b)=\del_i(b^\vee),\quad\text{and}\quad
(\e_ib)^\vee=\E_i(b^\vee)\quad\text{for all $b\in \B$ and $i\in I$.}
\eneqn
Here we understand $0^\vee=0$.
\ee
\enprop
\Proof
Assume first that $\B^\vee$ is a perfect basis.
Then, $\ker e_i^n=\soplus_{b^\vee \in \B^\vee ,\;\del(b^\vee)<n}\cor b^\vee$.
Since $f_i^nV$ coincides with the orthogonal complement
$(\ker e_i^n)^\perp\seteq\{u\in V\mid \lan u,\ker e_i^n\ran=0\}$
of $\ker e_i^n$,
we have
$f_i^nV=\soplus_{\del_i(b^\vee)\ge n}\cor b$.
Hence $b\in f_i^nV$ if and only if
$\del_i(b^\vee)\ge n$. Therefore we have
$\el_i(b)=\del_i(b^\vee)$, and
$$f_i^nV=\soplus_{\el_i(b)\ge n}\cor b.$$

We define $\f_i\cl \B\to\B\cup\{0\}$ by
$(\f_ib)^\vee=\F_i(b^\vee)$.
Let $b\in \B$ with $\el_i(b)=n$.
We shall show that
$f_ib-c\f_ib\in f_i^{n+2}V$ for some $c\in \cor^\times$.

Recall that
the image of $\B_n^\vee\seteq\set{b^\vee\in\B^\vee}{\del_i(b^\vee)=n}$ forms a basis of
$\ker e_i^{n+1}/\ker e_i^n$.
Since $\ker e_i^{n+1}/\ker e_i^n$ is isomorphic to
the dual
$(f_i^nV/f_i^{n+1}V)^\vee$ of $f_i^nV/f_i^{n+1}V$,
the image of $\B_n$ forms a basis of $f_i^nV/f_i^{n+1}V$ dual to $\B_n^\vee$.
By the hypothesis, the map
$e_i\cl \ker e_i^{n+2}/\ker e_i^{n+1}\to\ker e_i^{n+1}/\ker e_i^n$
gives an injection $$\E_i\cl\set{b^\vee \in \B_{n+1}^\vee}{\E_i(b^\vee)\not=0}
\to  \B_n^\vee.$$
Hence $f_i\cl f_i^nV/f_i^{n+1}V\to f_i^{n+1}V/f_i^{n+2}V$ sends
$\set{b\in\B_n}{\F_i(b^\vee)\not=0}$ to $\B_{n+1}$ up to a constant multiple,
and sends $\set{b\in\B_n}{\F_i(b^\vee)=0}$ to $\{0\}$.

As a conclusion, for any $b\in \B_{n}$, we have
$f_ib\equiv c\,\mathbf{f}_ib \mod f_i^{n+2}V$ for some $c\in\cor^\times$.
Thus $\B$ is a dual perfect basis.

The converse can be proved  in a similar manner. \QED
Recall that the Khovanov-Lauda-Rouquier algebra $R$ and its
cyclotomic quotient $R^{\lambda}$ provide a categorification of
$U_q(\g)$ and $V_q^{\lambda}$, respectively. That is, we have
\begin{equation*}
\begin{aligned}
& U_q^{-}(\g) \cong \Q(q)\otimes_{\Z[q,q^{-1}]}[\mathrm{Rep}(R)] \cong
\Q(q)\otimes_{\Z[q,q^{-1}]}[\mathrm{Proj}(R)], \\
& V_{q}(\lambda) \cong \Q(q)\otimes_{\Z[q,q^{-1}]}[\mathrm{Rep}(R^{\lambda})]\cong
\Q(q)\otimes_{\Z[q,q^{-1}]}[\mathrm{Proj}(R^{\lambda})].
\end{aligned}
\end{equation*}
It was shown in \cite{KOP2, LV} that the isomorphism classes of
finite-dimensional graded irreducible modules form a perfect basis
of $\Q(q)\otimes_{\Z[q,q^{-1}]}[\mathrm{Rep}(R)]$ and
$\Q(q)\otimes_{\Z[q,q^{-1}]}[\mathrm{Rep}(R^{\lambda})]$,
respectively. Since the isomorphism classes of finite-dimensional
graded irreducible $R$-modules and those of finitely generated
graded projective  indecomposable   $R$-modules are dual to
each other with respect to the perfect pairing given by
\begin{equation*}
\begin{aligned}
& [\mathrm{Proj}(R)] \times [\mathrm{Rep}(R)] \longrightarrow {\mathbb A}, \\
& \  ([P],  [M]) \longmapsto   \dim_q  \Hom_{R}(P, M),
\end{aligned}
\end{equation*}
the following proposition follows immediately.

\begin{prop} \label{Prop:Cat-infty}
 The isomorphism classes of finitely generated graded projective
indecomposable modules form a dual perfect basis of
$\Q(q)\otimes_{\Z[q,q^{-1}]}[\mathrm{Proj}(R)]$ and
$\Q(q)\otimes_{\Z[q,q^{-1}]}[\mathrm{Proj}(R^{\lambda})]$.
\end{prop}


\begin{thebibliography}{00}

\bibitem{BK}  A. Berenstein, D. Kazhdan,
{\em Geometric and unipotent crystals. II: From unipotent bicrystals to crystal bases},
Contemp. Math. {\bf 433} (2007), 13--88.

\bibitem{HK} J. Hong,  S.-J. Kang, 
{\em  Introduction to Quantum Groups and Crystal Bases}, Graduate
Studies in Mathematics {\bf 42}, Amer. Math. Soc., Providence, 2002.

\bibitem{JKK} K. Jeong, S.-J. Kang, M. Kashiwara
{\em Crystal bases for quantum generalized Kac-Moody algebras},
Proc. Lond. Math. Soc. {\bf 90 (3)} (2005), 395--438.

\bibitem{KK} S.-J. Kang, M. Kashiwara,
{\em  Categorification of highest
weight modules via Khovanov-Lauda-Rouquier algebras}, Invent. Math.
\textbf{190} (2012), 699--742.

\bibitem{KKO}
S.-J. Kang, M. Kashiwara, S.-j. Oh,
{\em Categorification of highest
weight modules over quantum generalized Kac-Moody algebras},
Moscow Math. J. {\bf 13} (2013), 315--343.


\bibitem{KKP}
S.-J. Kang, M. Kashiwara, E. Park, {\em Geometric realization of
Khovanov-Lauda-Rouquier algebras associated with Borcherds-Cartan
data}, Proc. London Math. Soc. {\bf (3) 107} (2013), 907-231.

\bibitem{KOP1} S.-J. Kang, S.-j. Oh, E. Park,
{\em  Perfect bases for integrable modules
over generalized Kac-Moody algebras},
Algebr. Represent. Theory {\bf 14} (2011), 571--587.


\bibitem{KOP2}
\bysame, {\em Categorification of quantum generalized Kac-Moody
algebras and crystal bases}, Int. J. Math. {\bf 23} (2012), 1250116.


\bibitem{K2} \bysame,
{\em On crystal bases of the q-analogue of universal enveloping algebras},
Duke Math. J. {\bf 63} (1991), 465--516.

\bibitem{K4} \bysame,
{\em Global crystal bases of quantum groups},
Duke math. J. {\bf 69} (1993), 455--485.

\bibitem{KL1} M. Khovanov, A. Lauda,
{\em A diagrammatic approach to categorification of quantum groups I},
Represent. Theory {\bf 13} (2009), 309--347.

\bibitem{KL2}\bysame,
{\em  A diagrammatic approach to categorification of quantum groups II},
Trans. Amer. Math. Soc. {\bf 363} (2011), 2685--2700.

\bibitem{LV} A. Lauda, M. Vazirani,
{\em Crystals from categorified
quantum groups},
 Adv. Math. {\bf 228} (2011), 803--861.

\bibitem{R} R. Rouquier, {\em
 2 Kac-Moody algebras}, arXiv:0812.5023.

\bibitem{R2}
 \bysame, {\em Quiver Hecke algebras and $2$-Lie algebras},
arXiv:1112.3619v1.

\bibitem{VV} M. Varagnolo, E. Vasserot,
{\em  Canonical bases and KLR algebras}, J. reine angew. Math. {\bf
659} (2011), 67--100.

\end{thebibliography}
\end{document}